\documentclass[11pt,a4paper]{amsart}
\usepackage{verbatim}

\usepackage[utf8]{inputenc}
\usepackage[english]{babel}
\usepackage[dvipsnames]{xcolor}
\usepackage{hyperref}
\usepackage{tikz}
\usepackage{tikz-cd}

\usepackage{soul,color}
\usepackage[colorinlistoftodos]{todonotes}
\usepackage[width=\textwidth,font=footnotesize,labelfont=bf, aboveskip=30pt]{caption}
\usepackage[symbol]{footmisc}

\newcommand*{\fullref}[1]{\hyperref[{#1}]{\ref*{#1}. \nameref*{#1}}} 

\newcommand{\sO}{\mathcal{O}}
\newcommand{\OO}{\mathcal{O}}

\newcommand{\albdim}{\operatorname{alb} \dim}
\DeclareMathOperator{\Alb}{Alb}
\DeclareMathOperator{\alb}{alb}

\newcommand{\pp}{\mathbb P}
\newcommand{\C}{\mathbb C}

\usepackage[osf]{mathpazo}
\usepackage{charter}
\usepackage[a4paper, top=3.2cm, bottom=3.2cm,left=2.8cm, right=2.8cm, heightrounded,bindingoffset=0mm]{geometry}

\usepackage{hyperref}
\hypersetup{bookmarksnumbered=true, linkcolor=due, citecolor=Green, urlcolor=black,}

\usepackage{fancyhdr}

\usepackage{amsfonts}
\usepackage{amssymb}
\usepackage[bb=libus]{mathalpha}
\usepackage{mathbbol}
\usepackage{extarrows}
\usepackage{cleveref}
\usepackage{amsthm}
\usepackage{amscd}
\usepackage{amsmath}
\usepackage{mathtools}
\usepackage{mathrsfs}
\usepackage{float}

\usepackage{color}	
\definecolor{due}{RGB}{0,76,147}

\usepackage{graphicx}	
\usepackage{multicol}	
\usepackage{wrapfig}
\usepackage[english]{babel} 
\usepackage[utf8]{inputenc}
\usepackage[T1]{fontenc}
\usepackage{enumitem}

\usepackage{longtable}
\usepackage{cite}

\theoremstyle{definition}
\newtheorem{defi}{Definition}[section]
\theoremstyle{plain}
\newtheorem{thm}[defi]{Theorem}

\newtheorem*{thm*}{Theorem}

\newtheorem{prop}[defi]{Proposition}

\newtheorem{lemma}[defi]{Lemma}
\theoremstyle{remark}

\newtheorem{rmk}[defi]{Remark}
\newtheorem{ex}[defi]{Example}

\theoremstyle{definition}

\newcommand{\inv}{^{-1}}

\newcommand{\Z}{{\mathbb Z}}

\newcommand{\Pic}{\operatorname{Pic}}

\newcommand{\ol}{\overline}
\newcommand{\wt}{\widetilde}

\newcommand{\fie}{\varphi}
\newcommand{\tp}{_{\rm top}}

  \makeatletter
\newcommand{\xdashrightarrow}[2][]{\ext@arrow 0359\rightarrowfill@@{#1}{#2}}
\newcommand{\xdashleftarrow}[2][]{\ext@arrow 3095\leftarrowfill@@{#1}{#2}}
\newcommand{\xdashleftrightarrow}[2][]{\ext@arrow 3359\leftrightarrowfill@@{#1}{#2}}
\def\rightarrowfill@@{\arrowfill@@\relax\relbar\rightarrow}
\def\leftarrowfill@@{\arrowfill@@\leftarrow\relbar\relax}
\def\leftrightarrowfill@@{\arrowfill@@\leftarrow\relbar\rightarrow}
\def\arrowfill@@#1#2#3#4{%
  $\m@th\thickmuskip0mu\medmuskip\thickmuskip\thinmuskip\thickmuskip
   \relax#4#1
   \xleaders\hbox{$#4#2$}\hfill
   #3$%
}
\makeatother

\theoremstyle{definition}
\newtheorem{Question}[defi]{Question}
\numberwithin{equation}{section}

\begin{document}

	\title{Numerical inequalities  for quasi-projective surfaces}
	\author{Rita Pardini} 
	\address{Dipartimento di Matematica\\ Universit\`a di Pisa, Largo B. Pontecorvo 5, 56127 Pisa, Italy}
\email{\url{rita.pardini@unipi.it}}

	\author{Sofia Tirabassi}
\address{Department of Mathematics\\ Stockholm University, Albano hus 1, Stockholm, Sweden}
\email{\url{tirabassi@math.su.se}}

\keywords{Quasi-projective  surfaces, Logarithmic plurigenus, Logarithmic irregularity, birational geometry of log-surfaces}
\subjclass[2020]{Primary 14J10, Secondary  14E05}
\thanks{Rita Pardini is a member of GNSAGA-INdAM}

\clearpage\maketitle
\begin{abstract}
Let $V$ be a smooth quasi-projective complex surface with compactification $(X,D)$ and set $\ol P_1(V):=h^0(X,K_X+D)$, $\ol q(V):=h^0(X,\Omega^1_X(\log D))$. We prove that $\ol P_1(V)\ge \ol q(V)-1$ if $V$ has maximal Albanese dimension and $\ol P_1(V)\ge\frac 16( \ol q(V)-5)$ otherwise.
Both bounds are sharp.
\end{abstract}

\setcounter{tocdepth}{1}
\tableofcontents
\section{Introduction}
 Let $X$ be   a smooth projective complex surface and let, as usual, $p_g(X): =h^0(X,K_X)$ be the geometric genus, $q(X):=h^0(X, \Omega^1_X)$ the irregularity and $\kappa(X)$ the Kodaira dimension. It is a classical result that if  $\kappa(X)\ge 0$ then $p_g(X)\ge q(X)-1$, and  if $p_g(X)=q(X)-1$ then   $X$ is birationally equivalent to one of the following:
 \begin{itemize}
  \item a bielliptic surface 
 \item an abelian surface
 \item a properly elliptic surface $X$ such that the elliptic fibration $f\colon X \to C$ is a quasi-bundle, namely all the  singular fibers of $f$ have smooth support. (cf. \cite[Prop.~4.2]{serrano-cetraro}).
  \end{itemize}
  In particular, if $p_g(X)=q(X)-1$ and $X$ has maximal Albanese dimension then $X$ is 
birationally either an abelian or a properly elliptic surface.  
  \medskip
  
    The above  numerical invariants can be defined  more generally in the logarithmic 
setting. Given  a smooth quasi-projective surface $V:=X\setminus D$, where $X$ is a smooth 
projective surface and $D\subset X$ is a snc divisor, one sets ${\ol 
P}_1(V):=h^0(X, K_X+D)$ ``first log plurigenus'', $\ol q(V):=h^0(X, \Omega^1_X(\log D))$ 
``log irregularity'',  and lets $\ol k(V)$ ``log Kodaira dimension'' be the Iitaka 
dimension of $K_X+D$. One  also defines (cf. \S \ref{ssec: invariants}) the (logarithmic) 
Albanese map $a_V\colon V \to \Alb(V)$, with $\Alb(V)$ a quasi-abelian variety.
  
   In this paper we establish sharp lower bounds for $\ol{P}_1(V)$ in terms of $\ol 
q(V)$ when $V$ has nonnegative logarithmic Kodaira dimension, and we give restrictions 
on the possible boundary examples. 

More precisely, we prove that when $D>0$ the inequality 
\begin{equation}\label{eq:main1}
 {\ol P}_1(V)\geq \ol q (V)-1
\end{equation}
still holds for  surfaces $V$ of maximal Albanese dimension, but not for surfaces 
with $\operatorname{alb} \dim =1$, indipendently of the logarithmic Kodaira dimension. In addition, we show that the
  open surfaces of maximal Albanese dimension for which  \eqref{eq:main1} is an equality satisfy strong geometric restrictions, and we 
provide  explicit examples, including  some of log general type.\par

For surfaces of $\operatorname{alb} \dim =1$ and nonnegative  log Kodaira dimension, we 
show that the following inequality always holds
\begin{equation}
 \label{eq:main2}
 {\ol P}_1(V)\geq \frac{1}{6}(\ol q (V)-5)
\end{equation}

This bound is sharp, and there are families of surfaces with $\ol{\kappa}(V)=2$,  ${\ol 
P}_1(V)= \frac{1}{6}(\ol q (V)-5)$, and $\ol{q}(V)\rightarrow+\infty$. So under this respect open surfaces 
behave in a strikingly different way than projective ones.\par
\medskip

The proof of \eqref{eq:main1} relies on a classical argument involving (a 
logarithmic version of) Castelnuovo--de Franchis Theorem due to Bauer \cite{ingrid}:
the key point is that on a surface of maximal Albanese dimension one can construct global 2-forms by wedging global 1-forms.  \par
On the contrary, when the Albanese dimension is 1 there is no apparent connection between the existence of 
global  1-forms and that of global 2-forms. Indeed, the proof of \eqref{eq:main1} in the 
projective case involves Noether formula, which is not available in the open case.
The proof of \eqref{eq:main2} is  based instead on a delicate study of the geometry of the 
Albanese pencil. This study is our motivation for introducing the  notion of
$w$-fiber of a projective fibration of a surface, which is analyzed in detail in Section 
\ref{wfib} and  is, in our opinion, of independent interest. Given a fibration $X\rightarrow B$ with  general fiber of positive genus, we bound from below the holomorphic Euler  
characteristic  $\chi(X)$ in terms of  the   number of $w$-fibers, and 
characterize the fibrations attaining the bound.
\smallskip

  \textbf{Assumption:} We work over the field of complex number $\mathbb{C}$, all varieties are quasi-projective. We do not distinguish between divisors an line bundles on smooth projective varieties and use the additive and multiplicative notation interchangeably. 
  
  \section{Preliminaries}
  In this section we fix notation and terminology and we recall some general facts.
 
\subsection{Logarithmic invariants} \label{ssec: invariants}
Given  a smooth quasi-projective variety $V$ of 
dimension $n$, we can define its logarithmic (log) cohomological invariants by considering 
$V$ as an open subset of a smooth projective variety $X$ such that the complement 
$X\backslash V$ is a divisor $D$ with simple normal crossing (snc) support. In this case, 
we say that the pair $(X, D)$ is a smooth  compactification of $V$, or a compactification of 
$V$ with smooth boundary $D$, and we can consider the sheaf $\Omega_X^1(\log D)$, which is 
the subsheaf of the sheaf of rational 1-forms on $X$ consisting of forms that admit  at most logarithmic poles along $D$. Since $V$ is smooth and $D$ is snc, this is a vector bundle of rank 
$n$ on $X$, with  top wedge power $\bigwedge^n\Omega_X^1(\log D)\simeq \omega_X(D)$, 
where $\omega_X=\OO_X(K_X)$ denotes, as usual, the canonical bundle of $X$. While these sheaves 
depend on the choice of compactification, their cohomology depends only on $V$. In particular we denote:
\begin{itemize}
 \item $\overline{q}(V):=h^0(X,\Omega_X^1(\log D))$ is the log irregularity of $V$;
 \item $\overline{P}_m(V):=h^0(X,\omega_X(D)^{\otimes m})$ is the $m$-th log plurigenus 
of $V$;
\item since $\omega_X(D)$ is a line bundle, we can consider its Iitaka dimension $\kappa 
(X,\omega_X(D))$. This is called the log Kodaira dimension of $V$ and denoted by 
$\overline{\kappa}(V)$.
\end{itemize}
When $\dim V=1$ we write $\ol g(V)$ instead of $\ol P_1(V)$.

\subsection{(Quasi)-Albanese varieties} It is well known that to any smooth 
projective variety we can associate a pair $(\Alb(X), \alb_X)$ with $\Alb(X)$ an 
abelian variety and $\alb_X\colon X\rightarrow \Alb(X)$ a morphism obtained by 
integrating $1$-forms, 
satisfying the obvious universal property. Deligne 
\cite[Corollaire (3.2.14)]{Deligne1971} has proven that log 1-forms are closed and,  similarly 
to 
what happens for smooth projective varieties,   the torsion free part of 
$H_1(V,\mathbb{Z})$  embeds into the complex vector space $H^0(X,\Omega_X^1(\log D))^{\vee}$  as a discrete subgroup. The 
quotient $H^0(X,\Omega_X^1(\log D))^{\vee}/H_1(V,\mathbb{Z})$ is an algebraic group denoted by 
$\Alb(V)$, but it is not in general an abelian variety, since the rank of 
$H_1(V,\mathbb{Z})$ might not be maximal. However $\Alb(V)$ sits in an exact sequence of 
algebraic groups as follows:
\begin{equation*}
 1\rightarrow \mathbb{G}_m^r\longrightarrow \Alb(V)\longrightarrow\Alb(X)\rightarrow 0.
\end{equation*}
Since $\Alb(X)$ is an abelian variety,  $\Alb(V)$ is a quasi-abelian 
variety. After choosing a base point, analogously to the projective setting, there is a 
natural morphism
$$\alb_V\colon V\rightarrow \Alb(V)$$
  such that any morphism $g\colon V\rightarrow G$ to a quasi-abelian variety factors uniquely 
through $\alb_V$. The pair $(\Alb(V),\alb_V)$ is called quasi-Albanese (or simply Albanese) variety of $V$.\par
The Albanese dimension of $V$ is the dimension of the image of $\alb_V$ and it is denoted 
by $\albdim(V)$. We say that $V$ has maximal Albanese dimension if $\albdim(V)=\dim V$. 
We also say that $V$ is of Albanese general type if it has maximal Albanese dimension and 
$\dim V<\overline{q}(V)$; in particular the Albanese map of $V$ is not dominant.\par
For further reference, we recall how one can compute the logarithmic irregularity of a
variety $V$ starting from the irregularity of a compactification $X$ and the 
configuration of the components of the boundary divisor $D$.
Denoting by $D_1,\dots D_k$ the components of $D$,  we have  the cohomology exact 
sequence
\[0\rightarrow H^0(X,\Omega_X)\overset{\rho}{\rightarrow} H^0(X,\Omega_X(\log D))\rightarrow \oplus 
H^0(X,\sO_{D_i})\xrightarrow{p}H^1(X,\Omega_X)\]
where $\rho$ is the residue map and  $p$ is the map sending $1\in H^0(X,\sO_{D_i})$ to $c_1(D_i)$, $i=1,\dots k$.
So we have that 
\begin{equation}\label{eq:qbar}
 \overline{q}(V)=q(X)+k-\operatorname{rk} p.
\end{equation}

Another result that we will use frequently is the following Theorem of Kawamata 
(\cite[Theorem 1]{Kawamata1981_Curves})
\begin{thm}[Kawamata subadditivity formula]\label{subadd}
 Let $f \colon  V \rightarrow W$ be a dominant
morphism of complex algebraic varieties. If the general fiber of $f$ is an irreducible 
curve $F$, then we have the following inequality for logarithmic Kodaira
dimensions:
$$\ol\kappa(V)\geq \ol\kappa(W ) + \ol\kappa(F).$$
\end{thm}

\subsection{Irrational pencils}\label{ssec: irr pencils}

Let $V$ be a smooth quasi-projective variety. An irrational pencil of $V$ of genus $\gamma>0$ is a morphism $\fie\colon V\to C$, where $C$ is a smooth curve with $\ol g(C)=\gamma$ such that the general fiber $\Phi$  of $\fie$ is irreducible.

We recall the following generalization of the Castelnuovo--de Franchis theorem:
\begin{thm}[\!\cite{ingrid}, Thm.~1.1]\label{thm: log CdF}
Let $V=X\setminus D$ be a smooth quasi-projective variety and let $\eta_1, \eta_2\in H^0(\Omega^1_X(\log D))$ be independent forms. 
If $\eta_1\wedge \eta_2=0$, then there are an irrational pencil $\fie\colon V\to C=\ol C\setminus \Delta$ and forms $\ol\eta_1,  \ol\eta_2\in H^0(\omega_{\ol C}(\Delta))$ such that $\eta_i=\fie^*\ol \eta_i$, $i=1,2$.
\end{thm}
A subspace $T$ of $H^0(\Omega^1_X(\log D))$ is called {\em totally isotropic} if $\dim T\ge 2$ and $\wedge^2 T=0$. A straightforward  consequence of Theorem \ref{thm: log CdF}   is that if $\fie \colon V\to C=\ol C\setminus \Delta$ is an irrational pencil of genus $\ge 2$, then $\fie^*H^0(\omega_{\ol C}(\log \Delta))$ is a maximal totally isotropic subspace and all maximal isotropic subspaces arise in this way. 

\subsection{WWPB equivalence}\label{ssec: WWPB}
When studying the birational geometry of open varieties, one first needs to understand what is 
the correct equivalence relation to use. In fact, the usual birational equivalence is 
quite meaningless in the open settting, since birational maps do not preserve logarithmic 
invariants.\par
Open immersions $i:U\rightarrow W$ where $W\backslash U$ has codimension at least 2, and 
proper birational maps preserve the logarithmic invariants. Thus one might be tempted to 
consider two varieties $V$ and $W$ to be equivalent if they are connected by a sequences 
of rational maps

\[
\begin{tikzcd}
V=:V_0 \arrow[r, dashed, "\varphi_0"] & V_1 \arrow[r, dashed, "\varphi_1"] & \dots 
\arrow[r, dashed, "\varphi_{n-1}"] & V_n:=W
\end{tikzcd}
\]
where the $\varphi_i$ are  proper birational, or open immersion as above, or their birational 
inverses. In this case we say that $V$ and $W$ are \emph{weak proper birationally 
(WPB) equivalent}.\par
Hoever, the example below show that the set of birational maps inducing WPB equivalences 
is not saturated, and thus there is a wider class of maps that preserves the logarithmic 
invariants.
\begin{ex}\label{ex;non saturated}
Let $X$ be a proper variety of dimension at least 2. Let $p\in X$ a point over 
$\mathbb{C}$ and consider $\operatorname{Bl_p}\colon \tilde{X}\rightarrow X$ the blow up of $X$ 
in $p$. Denote by $E$ the exceptional divisor, and let $V:=\tilde{X}\backslash E$ and let 
$U:=X\backslash p$. Now the restriction $\varphi:=\operatorname{Bl}_{p|V}\colon V\rightarrow X$ 
is clearly a WPB-morphism, since it is the composition of 
$\operatorname{Bl}_{p|V}\colon V\rightarrow U$ which is proper birational and the inclusion 
$U\hookrightarrow X$, which is a map of type (2) above.\par
However we can factor $\varphi$ as a different composition: we can first consider 
$\varphi_1:V\hookrightarrow\tilde{X}$ the inclusion, and then compose it with 
$\operatorname{Bl_p}\colon\tilde{X}\rightarrow X$. We note immediately that $\varphi_1$ is not 
WPB-morphism.
\end{ex}
 To fix the problem brought to light by the eample  we will consider a larger 
equivalence relation. Let 
\[\mathcal{W}:=\{f\colon V\rightarrow U\:|\:\exists\:W, \text{and 
either 
$g\colon U\rightarrow W$ or $h\colon W\rightarrow V$ such that $f\circ h$ 
or $g\circ f$ are WPB}\}\]
We say that a birational map $\varphi\colon V\dashrightarrow U$ is \emph{weakly weak proper 
birational (WWPB)} if it is the composition of morphisms in $\mathcal{W}$ or of rational 
inverses of morphism in $\mathcal{W}$.\par
It is easy to see now that, in the notation of example \ref{ex;non saturated}, we have 
that the inclusion $V\hookrightarrow\tilde{X}$ is a WWPB-morphism even if it was not WPB. 
We say that two varieties are WWPB-equivalent if there is a WWPB-map between them. In 
\cite{Iitaka WWPB} the author shows that WWPB equivalent varieties share the same 
logarithmic invariants.\par
From the work of Iitaka, it is commonly agreed that WWPB-equivalence is the correct 
notion to use when studying the birational geometry of open varieties.

  \section{The inequality ${\ol P}_1\ge \ol q-1$ for surfaces of maximal Albanese dimension }
  Our first result is a straightforward consequence of   the generalized Castelnuovo--de Franchis theorem (Theorem \ref{thm: log CdF}).
  \begin{thm}\label{thm: mAd} Let $V$ be a smooth quasi-projective surface of maximal Albanese dimension. Then:
  $${\ol P}_1(V)-\ol q(V)+1\ge 0.$$
  \end{thm}
  \begin{proof}  Let $(X,D)$ be a compactification of $V$. The maximal totally isotropic subspaces of $H^0(\Omega^1_X(\log D))$ correspond to the irrational  pencils of $V$ of genus $\ge 2$ (cf. \S \ref{ssec: irr pencils}), hence they are  proper subspaces, since $V$ has maximal Albanese dimension by assumption, and there are at most countably many, since they are mixed Hodge substructures of $H^1(V,\C)$. So, if $\eta\in H^0(\Omega^1_X(\log D))$ is very general, then the  kernel of the map $\wedge \eta \colon H^0(\Omega^1_X(\log D))\to H^0(K_X+D)$ is the line spanned by $\eta$. Therefore $\ol q(V)-1\le \ol P_1(V)$.
   \end{proof}

 \begin{rmk}\label{rem: 1sharp}
 Theorem \ref{thm: mAd} is sharp. Indeed, if $\Phi$, $C$ are quasi-projective curves  with $\ol g(\Phi)=1$ and $\ol g(C)>0$, then $V:=\Phi\times C$ has maximal Albanese dimension and $\ol P_1(V)=\ol g(C)=\ol q(V)-1$. If $\Phi$ is the complement of a point in a curve $\ol\Phi$ of genus 1, then $V$ is of log general type, in contrast with the projective case since  surfaces of general type satisfy $p_g>q-1$.

 \end{rmk}
 As recalled in the introduction,   a projective surface of Albanese general type with $p_g=q-1$ is an elliptic quasi-bundle on a curve of genus $q-1$. 
 Below we show the existence of a fibration onto a curve of genus $\ol q-1$ also in the logarithmic case. 
 
   \begin{prop}
  Let $V$ be a smooth quasi-projective surface of  Albanese general type such that $\ol q(V)\ge 3$ and ${\ol P}_1(V)-\ol q(V)+1=0$.\par  Then there is a fibration  $\fie \colon V\to C$   such that 
   $\ol g(C)=\ol q(V)-1$.
  \end{prop}
   \begin{proof} 
 Let $Y$ be a compactification of $\Alb(V)$ and let $(X,D)$ be a compactification of $V$ such that the Albanese map of $V$ extends to a morphism $\alpha\colon X\to Y$.   
 The proof is split in two steps. 
 \smallskip 
 
 We denote by $\Sigma\subseteq H^0(\Omega^1_X(\log D))$ the set of  1-forms that are not pull backs via an irrational pencil of genus $\ge 2$.  
   
 \noindent  {\bf Step 1:} {\em If $\eta\in  \Sigma$ and    $\eta(p)=0$ for some $p\in V$, then $p$ is a base point of $|K_X+D|$.}
 
 As explained in the proof of Theorem \ref{thm: mAd},   if $\eta\in\Sigma$  then  the kernel of $\wedge \eta \colon H^0(\Omega^1_X(\log D))\to H^0(K_X+D)$ is the line spanned by $\eta$.
Since $\ol P_1(V)=\ol q(V)-1$,  it follows that any $2$-form $\sigma\in H^0(K_X+D)$ can be written as $\sigma=\eta\wedge \beta$ for some $\beta\in H^0(\Omega^1_X(\log D))$, and therefore  vanishes at $p$.
  \medskip 

 \noindent {\bf Step 2:} {\em there is an irrational pencil $\fie \colon V\to C$ with $\ol g(C)=\ol q(V)-1$.}
 
  Let $p\in V$ be a very general point: then $p$ is not a critical point neither of 
$\alb_V$ nor of any irrational pencil of $V$. Let $T$ be the tangent space to $\Alb(V)$ at 
the origin and let $S$ be the 2-dimensional subspace of $T$ corresponding to the image of 
the differential of $\alb_V$ at $p$. Since  for any irrational pencil $\fie$ the plane $S$ 
intersects the tangent space to $\ker \fie_*$ in a line and since there are at most 
countably many such fibrations,  we can find a linear form  on $T$ that vanishes on $S$  
but does not vanish on the tangent space to $\ker \fie_*$ for any $\fie$. Pulling back 
this form to $V$ we obtain a 1-form $\eta\in\Sigma$ that vanishes at $p$.  By  Step 1, $p$ 
is a base point of $|K_X+D|$, contradicting the assumption that $p$ is very general.
  \end{proof}
\begin{Question}
In the examples given in Remark \ref{rem: 1sharp} the irrational pencil is the projection $\Phi\times C\to C$ and has general fiber with $\ol g =1$.
 We have not been able to find  examples  where the  general fiber of the irrational pencil has $\ol g>1$.
  Thus it is natural to wonder whether, as in the projective case,  the general fiber has always  $\ol{g}=1$.
\end{Question}

  \subsection{Surfaces  of maximal Albanese dimension and log general type with ${\ol 
P}_1(V)=1$, $\ol q(V)=2$}

We have no structure result for surfaces of log Albanese general type with  ${\ol P}_1(V)=1$, $\ol q(V)=2$.
In this section we describe  some  examples that show that a classification is probably hard to obtain.

\begin{ex} \label{ex: theta}
Let $X$ be a principally polarized abelian surface and $D$ a Theta divisor and set $V:=X\setminus D$. If $D$ is irreducible then $D$ is a smooth curve of genus 2 while if $D$ is reducible then it is the union of two elliptic curve meeting transversally at one point. 
In either case $D$ is snc and $\ol P_1(V)=1$, $\ol q(V)=2$, $\ol \kappa(V)=2$.
\end{ex}
Looking at the above example one might be tempted to conjecture that the only surfaces 
of 
maximal Albanese dimension and log general type with $\ol q=2$ and ${\ol P}_1=1$ are, up to WWPB-equivalence (cf. \S \ref{ssec: WWPB}),
complements of a theta-divisor in a principally polarized abelian surface, but our next examples show that the situation is 
more complicated.

\begin{ex}[$X$ abelian] \label{ex: theta gen}  
Example  \ref{ex: theta} generalizes as follows. Let $C\subset J(C)$ be a smooth genus two curve 
embedded in its Jacobian and let $G<J(C)$  be a finite subgroup of order $m$. Consider 
the 
quotient map  $p\colon J(C)\to \ol X:=J(C)/G$ and set  $\ol D:=p(C)$. By the Hurwitz 
formula $C$ does not have a fixed point free automorphism, so the map $C\to \ol D$ is 
birational. Let $\epsilon\colon  X\to \ol X$ be an embedded resolution of the 
singularities of $\ol D$,  denote by $D$ the strict transform of $\ol D$, which is a 
smooth curve of genus two and set $V:=X\setminus D$. Then taking cohomology in the 
residue 
sequence $0\to K_X\to K_X+D\to K_D\to 0$ and observing that the map $H^0(K_D)\to 
H^1(K_X)$ 
is dual to the map $ H^1(\OO_X)\to H^1(\OO_D)$, which is an isomorphism by construction, 
one sees that the natural map $H^0(K_X)\to H^0(K_X+D)$ is an isomorphism, hence ${\ol 
P}_1(V)=1$. 
Since $D$ is irreducible, the map $H^0(\Omega^1_X)\to H^0(\Omega^1_X(\log D))$ is an 
isomorphism, hence  $\ol q(V)= q(X)=2$ and $X$ is the Albanese map of $V$.

Finally we need to show that $V$ is of log general type. Denote by $E_1,\dots E_k$ the 
(possibly reducible) $(-1)$-curves  of $X$ and write $D=\epsilon^*\ol D-\sum m_iE_i$, so 
that $K_X+D=\epsilon^*\ol D-\sum (m_i-1)E_i$. For $0<t\ll 1$, $K_X+D=t\epsilon ^*\ol 
D+(1-t)(\epsilon^*\ol D-\sum \frac{m_i-1}{1-t}E_i)\ge t\epsilon ^*\ol D+(1-t)D$, so 
$K_X+D$ is big and $\ol \kappa(V)=2$.
\end{ex}
Next we give an example where $X$ is a properly elliptic surface. 
\begin{ex}[$X$ properly elliptic] \label{ex: ell}
Let $E$ be a genus one curve and $\pi_1\colon C\to E$ a double cover, with $C$ smooth of 
genus 2. Denote by $\iota$ the hyperelliptic involution of $C$ and by $\sigma$ the 
involution of $C$ induced by $f$. Since $\iota$ commutes with all the automorphisms of 
$C$, the group generated by $\iota$ and $\sigma$ is isomorphic to $\Z_2^2$ and 
$\tau:=\iota\circ \sigma$ is an elliptic involution. 
We denote by $\pi_2\colon C\to F:=C/\tau$ the quotient map; up to a translation in $F$, 
we 
may assume that $\pi_2(\sigma(x))=-\pi_2(x)$ for all $x\in C$. We let $\Z_2$ act on $C$ 
as 
$\sigma$, on $F$ as translation by a nonzero  $2$-torsion point $\eta$ and diagonally on 
$C\times F$ and we denote by $S$ the quotient surface.  Since  $\Z_2$ acts freely, the 
quotient map $C\times F\to S$ is \'etale, hence $S$ is a smooth minimal properly elliptic 
surface. In addition, $q(S)=2$, $p_g(S)=1$; more precisely  $\Alb(S)=E\times F'$, where 
$F'$ is the quotient of $F$ by translation by $\eta$,  and the Albanese map is the double 
cover of $E\times F'$ induced by $C\times F\to E\times F'$.

We let $f\colon S\to E$ and $h\colon S\to F'$ be the fibrations induced by the 
projections 
of $C\times F$ onto the factors. The general fiber of $f$ is isomorphic to $F$, there 
are 
two double fibers $2F_1$ and $2F_2$ and $F_1+F_2$ is the only effective canonical 
divisor. The fibration $h$ is smooth and isotrivial with  fibers isomorphic to $C$.

 One can define two additional  fibrations $\psi^+, \psi^-\colon C\times F\to F$ by 
setting $(x,y)\mapsto y+\pi_2(x)$, $(x,y)\mapsto y-\pi_2(x)$, respectively.  The 
diagonal 
$\Z_2$-action on $C\times F$ switches $\psi^+$ and $\psi^-$, and 
 both fibrations are smooth and isotrivial with fibers isomorphic to $C$. Given $z\in F$, 
we denote by $G^+_z$, resp. $G^-_z$ the fiber over $z$ of $\psi^+$, resp. $\psi^-$.
 So $G^+_z$ is defined by  $y+\pi_2(x)=z$ and its transform under  the $\Z_2$-action is 
$G^-_{z+\eta}$, defined by  $y-\pi_2(x)=z+\eta$. The common points of $G^+_z$ and 
$G^-_{z+\eta}$ satisfy $2y=2z+\eta$. So there are 4 possibilities for $y$ and, if $z$ is 
general, $G^+_z$ and $G^-_{z+\eta}$ meet transversally at 8 points. Set $\Gamma_0:= 
G^+_z+G^-_{z+\eta}$: $\Gamma_0$ is nef and big, $\Gamma_0^2=16$, $\Gamma_0K_{C\times 
F}=4$. So the image $\Gamma$ of $\Gamma_0$ in $S$ has 4 nodes and satisfies $\Gamma^2=8$, 
$\Gamma K_S=2$. 
 Since $\Gamma$ is nef and big, $h^1(-\Gamma)=0$ by Kawamata-Viehweg vanishing, hence the 
map $H^1(\OO_S)\to H^1(\OO_{\Gamma})$ is injective. So the natural map  $\Pic^0(S)\to 
\Pic^0(\Gamma)$ has finite kernel. Since $\Gamma$ is birational to $C$, the compact part 
of   $\Pic^0(\Gamma)$ is the Jacobian $J(C)$ and therefore there is an induced isogeny 
$\Pic^0(S)\to J(C)$.
  
 Now let $\epsilon\colon X\to S$ be the blo$w$-up of the 4 nodes of $\Gamma$ and let $D$ be 
the strict transform of $\Gamma$ in $X$. The divisor $D$ is smooth, isomorphic to $C$,  
and there is a commutative diagram:
 \[ 
   \begin{tikzcd}
   D \rar \dar & \Gamma \dar\\
   \Alb(X) \rar & \Alb(S)
   \end{tikzcd}
  \]
  where the bottom arrow is an isomorphism. By the previous considerations, the induced 
map $J(D)\to \Alb(X)$ is an isogeny, hence $H^1(\OO_X)\to H^1(\OO_D)$ is an isomorphism. 
Hence, by Serre duality the dual map $H^0(K_D)\to H^1(K_X)$ is also an isomorphism and 
taking cohomology in  the residue sequence $0\to K_X\to K_X+D\to K_D\to 0$ gives an 
isomorphism $H^0(K_X)\cong H^0(K_X+D)$, hence $h^0(K_X+D)=1$. Set $V:=X\setminus D$:  we 
have just seen ${\ol P}_1(V)=1$ and $\ol q(V)=2$ by \eqref{eq:qbar}, since $D$ is irreducible. To see 
that $V$ is of log general type observe that $m(K_X+D)\ge mD+m\epsilon^*K_S\ge 
D+m\epsilon^*K_S$. 
  Since $(\epsilon^*K_S)D=K_S\Gamma=2$ and $D^2=-8$, for $m\ge 5$ the divisor 
$D+m\epsilon^*K_S$ is nef and big. It follows that $\ol\kappa(V)=2$. 
  
  Finally note that a general fiber of the elliptic fibration $X\to E$ meets $D$ in 2 
points, hence the general fiber of the induced  fibration $V\to E$  has $\ol g=2$.
  \medskip
  
  \end{ex}

  \section{$w$-fibers}\label{wfib}
In this section we explore the notion of $w$-fiber  (``wicked'' or ``weird'' fiber), 
which is key in the proof of the inequalities of \S \ref{sec: alb1}.

Let $X$ be a smooth projective surface and let $f\colon X\to B$ a fibration onto a  curve of genus $b\ge 0$ with  general fiber of  genus $g>0$. We say that a fiber  $\wt F$  of $f$ is a $w$-fiber if it is supported on a divisor $F$ such that $p_a(F)=0$. Equivalently, $\wt F$ is a $w$-fiber if all the components of $\wt F$ are rational curves,  the support $F$ of $\wt F$  is snc and its  dual graph is a tree.

\begin{rmk}\label{rem: w elliptic}
If $f$ is a relatively minimal elliptic fibration then the $w$-fibers are the fibers of type $I_n^*$ for $n\ge0$,  $II^*$, $III^*$, $IV^*$. If $f$ has fibers  of type $II$, $III$, $IV$ then by taking a suitable sequence of  blo$w$-ups one can  obtain a fibration $f'\colon X'\to B$ such that the pull backs of these fibers are $w$-fibers.
\end{rmk}

\begin{prop} \label{prop: w fibers}
Let $f\colon X\to B$ be a fibration of a smooth projective surface onto a smooth curve   with general fiber of genus $g>0$  and let $\nu$ denote the number of $w$-fibers of $f$. If  $\nu>0$, then:
 \begin{enumerate} 
 \item $g(B)=q(X)$ and $\chi(X)\ge \frac{1}{6}\nu$.
\item if  $\chi(X)= \frac{1}{6}\nu$ then:
 
\begin{itemize}
\item[(a)] $g=1$, $f$ has constant moduli  and the smooth fibers of $\ol f$ are isomorphic to the  elliptic curve $E$ with affine equation  $y^2=x^3+1$;
\item[(b)]  the   fibers of  the relative minimal model $\ol f\colon \ol X\to\pp^1$ of $f$ with singular support are  $\nu$  fibers of type $II$.
\end{itemize}
\end{enumerate}
\end{prop}
The example that follows shows that  surfaces as in   Proposition \ref{prop: w fibers}, (2),  actually exist. 
\begin{ex} \label{ex: elliptic fiber}
Let $d:= 6k>0$ be an integer, and let $p_1, \dots p_d\in \pp^1$ be distinct points. Let $C\to \pp^1$ be the simple cyclic cover of degree 3 branched on $p_1+\dots+p_d$. The curve $C$ can be embedded in $\mathbb F_{2k}$ as an element of $|3\sigma_0|$, where $\sigma_0$ is a section with $\sigma_0^2=2k$.
Let $\pi \colon \ol X\to \mathbb F_{2k}$ be the double cover branched on $C+\sigma_{\infty}$, where $\sigma_{\infty}$ is the section at infinity, and let $\ol f\colon X\to \pp^1$ be the fibration induced by the projection $\mathbb F_{2k}\to \pp^1$. The surface $\ol X$ is smooth and using the standard formulae for double covers one obtains $\chi(X)=k$, $q(X)=0$. The fibration $\ol f$ is elliptic and has precisely $d=6k$  singular fibers, all  of type $II$, occurring above the points $p_1,\dots p_d\in \pp^1$. Taking a log resolution of the singular fibers one obtains a fibration $f\colon X\to\pp^1$ with $d=6k$ $w$-fibers

The automorphism of $C$ of order $3$ is the restriction of an automorphism $\psi$ of $\mathbb F_{2k}$ that fixes  $\sigma_{\infty}$ pointwise. So the general fiber of $\bar f$ is a double cover of $\pp^1$ branched over 4 points that are  invariant under an automorphism of $\pp^1$ of order 3, so it is isomorphic to  the elliptic curve $E$ with affine equation  $y^2=x^3+1$, the only elliptic curve  that has an automorphism of order 3, as predicted by Proposition \ref{prop: w fibers}.
In this construction one can replace $\pp^1$ by any smooth curve $B$, obtaining analogous 
examples with $q(X)=g(B)$ and $\chi(X)=\frac 16 \nu$. \par
It is not difficult to give an alternative description of $X$ as the minimal resolution 
of  the quotient of a product of two curves.
\end{ex}

The next example shows that for every $g\ge 2$ there are fibrations such that all their singular fibers are $w$-fibers. 
\begin{ex}
This is a variation of the previous example. Let $g\ge 2$, $d=4g+4$, let $p_1,\dots p_d\in 
\pp^1$ be distinct points and let $C\rightarrow \pp^1$ be the simple cyclic cover 
branched on $p_1+\dots +p_d$.  Then $C$ embeds into $\mathbb F_2$ as an element of 
$|(2g+1)\sigma_0|$, where $\sigma_0$ is a section with self-intersection 2.
Let $\pi \colon \ol X\to \mathbb F_{2}$ be the double cover branched on $C+\sigma_{\infty}$, where $\sigma_{\infty}$ is the section at infinity, and let $\ol f\colon X\to \pp^1$ be the fibration induced by the projection $\mathbb F_{2}\to \pp^1$. The surface $\ol X$ is smooth;
 the fibration $\ol f$ is elliptic and has precisely $d$  singular fibers, all  homeomorphic to $\pp^1$, occurring above the points $p_1,\dots p_d\in \pp^1$. Taking a log resolution of the singular fibers one obtains a fibration $f\colon X\to\pp^1$ with $d$ $w$-fibers. Arguing as in the previous example one shows that all the smooth fibers of $f$ are isomorphic to the hyperelliptic curve with affine equation  $y^2=x^{2g+1}+1$
 
 Using the standard formulae for double covers one can check that $\chi(X)>\frac 16 d$, in accordance with Proposition \ref{prop: w fibers}.
\end{ex}

In the proof of Proposition \ref{prop: w fibers} we will use the following: 
\begin{lemma} \label{lem: isotrivial}
Let $Y$ a smooth projective surface and let $h\colon Y\to B$ be an elliptic fibration onto a smooth curve. If all the  fibers of  $f$ with singular support are of type $II$ and  there is at least one such fiber, then $h$ is isotrivial and the smooth fibers of $h$ are isomorphic to the   elliptic curve $E$ with affine equation  $y^2=x^3+1$.
\end{lemma}
\begin{proof}
This  result is certainly well known to experts, hence we just sketch the proof.

After making a ramified base change on $B$ and normalizing, we can assume that there is a section $ B\to X$  and hence that there are no multiple fibers. Note that, since  fibers of type $II$ cannot be multiple, the existence of fibers of type $II$ is preserved.
 After a further ramified base change, we assume that $h\colon X \to B$ has a stable model $Y\to B$, in which the singular fibers are replaced by stable 1-pointed curves. Then a standard calculation reveals that, since $X$ is smooth, the fiber of $Y\to B$ corresponding to a fiber of type $II$  curve in $h\colon X\to B$ is the elliptic curve $E$ with affine equation $y^2=x^3+1$; hence $Y\to B$ has smooth fibers. Since the moduli space of elliptic curves is $\mathbb A^1$, it follows that $Y\to B$ is isotrivial with fiber $E$.
\end{proof}

\begin{proof}[Proof of Prop. \ref{prop: w fibers}]
Write $q:=q(X)$.  The $w$-fibers of $f$ are contracted by the Albanese map $\alb_X$ of $X$, since they are supported on a union of rational curves. Hence all fibers of $f$ are contracted by $\alb_X$ and  the Albanese dimension of $X$ is at most 1. So  either $q=0$ or the fibration $f$ coincides with  the Albanese pencil. In either case $g(B)=q$. In particular, if $X$ is ruled then it is a rational surface.

Let $\ol f\colon \ol X\to B$ be the relative minimal model of $f$. The morphism $\epsilon \colon X\to \ol X$ factorizes as $X\to X'\to \ol X$, where the   first map contracts  precisely the $\epsilon$-exceptional curves not contained in a $w$-fiber of $f$. 
Let  $F_i$, $i=1,\dots \nu$,  be  the supports of the $w$-fibers of $f$ and  denote by $\rho_i$ the number of exceptional curves contained in $F_i$. So $F_i$ has at least $\rho_i+1$ components and $\chi\tp(F_i)\ge \rho_i+2$. Now the formula for the topological Euler characteristic of a fibration (\cite[Prop.~III.11.4]{bhpv}) gives:
\begin{equation}\label{eq: c2}
 c_2(X')\ge 4(g-1)(q-1)+\sum (\chi\tp(F_i) -(2-2g))\ge  4(g-1)(q-1)+\sum \rho_i+2g\nu.
 \end{equation}
Since $c_2(X')=c_2(\ol X)+\sum \rho_i$, this is the same as
$$
c_2(\ol X)\ge  4(g-1)(q-1)+2g\nu
$$
 If $g=1$, then $K_{\ol X}^2=0$, while for $g>1$ Arakelov's Theorem (\cite{appendix}) gives $K_{\ol X}^2\ge 8(q-1)(g-1)$. So, using Noether's formula, one obtains:
\begin{gather}\label{eq: arakelov} 
12(\chi(\ol X)-(q-1)(g-1))=(K_{\ol X}^2- 8(q-1)(g-1))+(c_2(\ol X)-4(q-1)(g-1)) \\
\ge c_2(\ol X)-4(q-1)(g-1)\ge 2g\nu, \nonumber
\end{gather}
namely 
\begin{equation}\label{eq: chi}
\chi(X)=\chi(\ol X)\ge \frac{g}{6}\nu+ (q-1)(g-1)
\end{equation}
If $g=1$ or  $q>0$ or  $q=0$ and $\nu\ge 6$, one has $\frac{g}{6}\nu+ (q-1)(g-1)\ge \frac 16\nu$, so \eqref{eq: chi} implies  $\chi(X)\ge \frac{1}{6}\nu$. 
If $g>1$,  $q=0$ and $\nu<6$, then $\chi(X)=1+p_g(X)\ge 1>\frac 1 6\nu$. This completes the proof of (1).
\medskip

Finally assume that $\chi(X)=\frac 16\nu$. Then  all the inequalities that appear in the proof are actually equalities. In particular, if \eqref{eq: c2} is an equality, then all the singular fibers of $\ol f$ that contribute to the computation of $c_2(\ol X)$ arise from the $w$-fibers of $f$,  are irreducible and have  $\chi\tp=2$.
If $g=1$, then all the  fibers of $\ol f$ with singular support are of type $II$ (cf. Remark \ref{rem: w elliptic}), hence  by Lemma \ref{lem: isotrivial} $X$ is as claimed.

Assume for contradiction $g\ge 2$. Since $\frac{g}{6}\nu+ (q-1)(g-1)\ge \frac{1}{6}\nu$  for $\nu\ge 6$, equation \eqref{eq: chi} implies  $\frac{g}{6}\nu+ (q-1)(g-1)= \frac{1}{6}\nu$, hence 	$q=0$. Therefore the restriction map $H^0(X, K_{\ol X}+\ol F)\to H^0(K_{\ol F})$ is surjective for any fiber $\ol F$ of $\ol f$. 		
   Since all fibers are irreducible, it follows that $|K_{\ol X}+\ol F|$ is base point free, and therefore $0\le (K_{\ol X}+\ol F)^2=K_{\ol X}^2+4g-4$, namely $K_{\ol X}^2\ge -4(g-1)$. On the other hand $K^2_{\ol X}= -8(g-1)$ because
 \eqref{eq: arakelov} is an equality, hence we have a contradiction that shows that $g>1$ does not occur.
 \end{proof}

\section{The inequality $\ol P_1\ge\frac16( \ol q-5)$  for surfaces of Albanese dimension 1}\label{sec: alb1}

Here we apply the results of the previous section and prove the following inequality for surfaces with Albanese dimension 1:

\begin{thm}\label{thm: albdim1}
Let $V$ be a quasi projective surface with $\operatorname{alb dim} (V)=1$ and $\ol\kappa(V)\ge 0$,
 then 
$$ \overline{P}_1(V)\geq  \frac{1}{6}(\overline{q}(V)-5).$$
If equality holds, then $q=0$ and the fibration $h\colon X \to\pp^1$ is as in Proposition \ref{prop: w fibers}, (2). In addition,  up to WWPB-equivalence   $D=F_1+\dots +F_r$ or $D=H+F_1+\dots +F_r+E$, where  $F_1, \dots F_r$ are the supports of the $w$-fibers of $h$,  $H$ is a section of $h$, and  $E$ is a vertical divisor that does not support a full fiber of $h$.
\end{thm}

Theorem \ref{thm: albdim1} is sharp, as shown by the next example.
\begin{ex} Let $\ol h\colon \ol X \to \pp^1$ be a fibration as in Example \ref{ex: elliptic fiber}, with $\chi(X)=k$ and $6k$ singular fibers homeomorphic to $\pp^1$,   let $X\to\ol X$ be a log resolution of the  singular fibers of $\ol h$ and let $h\colon X\to\pp^1$ be the induced fibration. Denote by $F_1,\dots F_{6k}$ the supports of the singular fibers of $h$ and let $\Delta$ be the strict transform on $X$ of the section  $\sigma_{\infty}$ of $\mathbb F_{2k}$.  Set $D=F_1+\dots +F_{6k}$ or $D=F_1+\dots +F_{6k}+\Delta$ and $V:=X\setminus \Delta$; in either case by \eqref{eq:qbar}  the Albanese pencil is the restriction of $h$ and $\ol q (V)=6k-1$. In addition, standard computations (cf. the proof of Theorem \ref{thm: albdim1} below) show $\ol P_1(V)=k-1$.

Let now $x\in \Delta$ is a point not on $F_1+\dots +F_{6k}$,  let $X'\to X$ be the blow-up at $X$ and let $E$ be the exceptional curve.
Denote by $\Delta'$ the strict transform of $\Delta$ and use the same letter for curves on $X$ and their total transforms on $X'$.
Set $V':=X'\setminus (\Delta'+F_1+\dots +F_{6k})$. Then $\ol q(V')=\ol q(V)=6k-1$ by construction  and $\ol P_1(V')\le \ol P_1(V)=k-1$. On the other hand, $\ol P_1(V')\ge k-1$ by Theorem \ref{thm: albdim1}, so $\ol P_1(V')=k-1$. Summing up, the surfaces $V$ and $V'$  both satisfy equality in Theorem \ref{thm: albdim1} but the natural map $V\to V'$ is a so-called ``half-point attachment'' (\cite[\S 2]{Iitaka1979}), which in general is  not a WWPB equivalence.
\end{ex}

\begin{proof}[Proof of Thm. \ref{thm: albdim1}]
The statement is void for $\ol q(V)\le 5$, so we assume $\ol q(V)\ge 6$.
If $D=0$ it is well known that 
\begin{align}
 \ol P_1(V)&=p_g(X)=\chi(X)+q(X)-1\notag\\
 &\ge q(X)-1=\ol q(X)-1\label{eq:prog}\\
 &>\frac16( \ol q(V)-5)\notag,
\end{align}
where \eqref{eq:prog} comes from the fact that $\kappa(X)\geq 0$. So we assume $D>0$.
\smallskip

Let $X$ be a compactification of $V$ with snc boundary $D$ such that the Albanese pencil 
$V\rightarrow C$ compactifies to a fibration $h\colon X\rightarrow \overline{C}$ and let 
$F$ be a general fiber of $h$.
 Then we 
can write 
$$
D=H+\sum_{i=1}^r F_i + E,
$$
where $H$ is $h$-horizontal, the $F_i$ are the (reduced) supports of whole fibers of $h$ and $E$ 
is an $h$-vertical divisor whose  connected components  do not support whole fibers. 
Let $V'$ be the surface $X\backslash (H+F_1+\cdots+F_r)$ obtained from $V$ by 
``reattaching'' $E$. Since $h$ induces a fibration $h'\colon V'\rightarrow C$, we have   $\ol g(C)=\ol{q}(V)\le\ol{q}(V')$ 
and therefore $\ol{q}(V')=\ol{q}(V)$.
The general fiber of $h'$ is the same as the general fiber of 
$\alb_V\colon V\rightarrow C$, which we denote by $F$. By Kawamata subadditivity (Theorem 
\ref{subadd}), we have that 
\[
 \ol{\kappa}(V')\geq\ol{\kappa}(C) +\ol{\kappa}(F)\geq0.
\]
In addition, 

$$h^0(K_X+D)\geq h^0\left(K_X+H+F_1+\cdots+F_r\right),$$ 
 thus we can assume that $E=0$.\par

Note that either $q(X)=g(\ol C)=0$ or $\albdim X=1$,    $h\colon X\rightarrow 
\overline{C}$ is  the Albanese pencil of $X$, and  $g(\overline{C})=q(X)\geq 1$.\par
 When $H\neq 0$, we write $H=H_1+\cdots+H_s$ with $H_i$ irreducible and set $m_i:=H_iF$,  
$m:=\sum m_i=HF$. Let $\psi_i\colon H_i\to C$ be the morphism induced by 
$h$:  the logarithmic ramification formula implies $K_{H_i}(F_1+\dots +F_r)\ge 
\psi_i^*(K_{\ol C}+\Delta)$,  hence $\deg(K_{H_i}(F_1+\dots +F_r))\ge  m_i(2q-2+r)$ for 
$i=1,\dots s$.	So 
\begin{align}
D(K_X+D)&=\sum_1^s H_i(K_X+D)+\sum_1^r F_i(K_X+D)\notag\\
 &\ge \sum_1^s \deg(K_{H_i}(F_1+\dots +F_r))+\sum_1^r \deg(K_{F_i}(H))\nonumber \\
 \end{align}
 Since $H\cdot F_i\geq 1$ we have that $\deg(K_{F_i}(H))\geq 2p_a(F_i)-2+1= 2p_a(F_i)-1$. 
In particular we get:
 \begin{align}\label{eq: D(K+D)}
D(K_X+D) &\ge m(2q(X)-2+r)+\sum_1^r(2p_a(F_i)-1).
 \end{align}
 Observe that 
 \begin{equation}
  2q(X)-2+r=\ol{q}(V)+q(X)-1.
 \end{equation}
 We suppose first that $g(F)>0$. In particular we have that $X$ is not an irrational 
ruled surface (or the general fiber of the Albanese pencil would be $\mathbb{P}^1$) and 
in particular $\chi(X)\geq 0$. Thus we have that the inequality of Proposition \ref{prop: 
w fibers}
$$
\chi(X)\geq \frac16\nu
$$
holds also when there are no $w$-fibers.\par
We consider first  the case where the divisor $D$ is vertical, that is  $H=0$.
 Considering  the cohomology sequence associated with the short exact sequence
$$
0\rightarrow \sO_X\left(-\sum F_i\right)\longrightarrow \sO_X\longrightarrow \oplus 
\sO_{F_i}\rightarrow 0
$$
we observe that  since the $F_i$ are 
contracted by the Albanese map of $X$,
the restriction map $H^1(\sO_X)\rightarrow H^1( \oplus 
\sO_{F_i})$ is zero. 
Therefore we have that 
$h^1(X, D)=q(X)+r-1=\overline{q}(V)$. Then, by Riemann-Roch, we have
$$h^0(X,K_X+D)=\chi(X)+\sum (p_a(F_i)-1)+ \overline{q}(V).$$

Let $0\leq t\leq r$ be the number of $F_i$ which support a  $w$-fiber.

By Proposition \ref{prop: w fibers},
we have that
\begin{equation*}
 h^0(X,K_X+D)\geq \chi(X) -t+ \overline{q}(V) \geq  -\frac{5}{6}t+ \overline{q}(V),
\end{equation*}
 Now we observe that 
$$\ol{q}(V)-\frac{5}{6}t=\frac56(r-t+q(X)-1)+\frac16(q(X)+r-1)\ge \frac 16 \ol 
q (V)-\frac56.$$
Therefore, the inequality is proven in this case.\par
Suppose next that $H>0$. In this case the map $H^1(\OO_X)\cong H^1(\OO_{\ol C})\to 
H^1(\OO_D)$ is injective, so  taking cohomology in the short exact sequence
$$
0\rightarrow \sO_X\left(-D\right)\longrightarrow \sO_X\longrightarrow 
\sO_D\rightarrow 0
$$
 we get $h^1(-D)=\rho-1$, where $\rho$ is the number of connected components of $D$.  Hence
\begin{equation}\label{eq: RR}
h^0(X, K_X+D)=\chi(X)+\frac12D(K_X+D)+\rho-1\geq \chi(X)+\frac12D(K_X+D),
\end{equation}
where equality holds exactly when $D$ is connected.

As before let $0\le t \le r $  be the number of the $F_i$ that support a $w$-fiber; 
equations  \eqref{eq: RR}, \eqref{eq: D(K+D)} and Proposition 
\ref{prop: w fibers} give:
 \begin{align}
h^0(X, K_X+D)&\ge \frac 16t+(q(X)-1)+\frac{1}{2}(r-t)\\\nonumber
&= \frac16(q(X)-1+r)+\frac13(r-t)+\frac 56(q(X)-1)\notag\\ 
&\ge \frac 16 \ol q(V)-\frac 56.
\end{align}

Now we turn our attention to the case in which $g(F)=0$. Observe that, since $\ol \kappa 
(V)\geq 0$, then $H\neq 0$,  and $m\ge 2$. Thus, 
\eqref{eq: D(K+D)} gives:
\begin{gather}
h^0(X, K_X+D)\ge 2q(X)-2+r -\frac 12r =\frac12\ol q(V)+\frac32(q(X)-1)\ge \frac12\ol 
q(V)-\frac32>\frac 16\ol q(V)-\frac 56
\end{gather}

The proof of the inequality is complete. Now assume that $\ol P_1(V)=\frac 16\ol 
q(V)-\frac 56$. Then all the inequalities used in the proof are equalities, and so, in 
particular:
\begin{itemize}
\item[(a)] $D>0$, $q(X)=0$, $g(F)>0$, $HF=0,1$; 
\item[(b)] $r-1=\ol q\ge 5$ and  $F_1,\dots F_r$ are precisely the  supports of the $w$-fibers of $h$ and $\chi(X)=\frac16 r$.
\end{itemize}
So  by Proposition \ref{prop: w fibers}, (2), the fibration $h\colon X\to \pp^1$ is as 
described therein,   $D=H+F_1+\dots +F_r+E$, where either $H=0$ or $H$ is a section of 
$h$, $F_1, \dots F_r$ are the supports of the $w$-fibers of $h$ and $E$ is a divisor 
contracted by $h$ that  does not support a full fiber. Since the only singular fibers of 
the relative minimal model $\ol h\colon \ol X\to \pp^1$ are those arising from the 
$w$-fibers, the connected components of $E$ are contracted by $ X\to \ol X$ to points 
lying on smooth fibers of $\ol h$ (cf. \ref{ssec: WWPB}). In particular if $H=0$, we may assume $E=0$ up to WWPB equivalence.
 \end{proof}

\end{document}